\documentclass[letterpaper,11pt,reqno,oneside]{amsart}

\usepackage[totalheight=9.6in,totalwidth=6.15in,centering]{geometry}
\usepackage[dvipsnames]{xcolor}
\usepackage[utf8]{inputenc}
\usepackage[T1]{fontenc}
\usepackage{lmodern,microtype}
\usepackage{upref}
\usepackage{mathtools,amssymb,amsthm,amsmath}
\usepackage{enumitem}
\usepackage{todonotes}
\usepackage{graphicx}
\usepackage{hyperref}
\usepackage{wrapfig}
\usepackage{array}
\usepackage[margin=1cm]{caption}
\usepackage{bookmark}
\usepackage{mycommands}
\usepackage{nicefrac}
\usepackage{dsfont}
\usepackage{comment}

\mathtoolsset{showmanualtags,showonlyrefs}

\usepackage{multicol}
\usepackage{tasks}

\graphicspath{{./graphics/}}

\newtheorem{thm}{Theorem}
\newtheorem{theorem}[thm]{Theorem}
\newtheorem{corollary}[thm]{Corollary}

\newtheorem{prop}[thm]{Proposition}
\newtheorem{lem}[thm]{Lemma}
\theoremstyle{remark}

\newtheorem{rem}{Remark}
\theoremstyle{definition}


\definecolor{darkblue}{rgb}{0.13,0.13,0.39}
\hypersetup{colorlinks=true,urlcolor=darkblue,citecolor=darkblue,linkcolor=darkblue,
  pdftitle={Restricted maximum of non-intersecting Brownian bridges},
  pdfauthor={Y. Yalanda, N. Zalduendo}
}

\usepackage{todonotes}
\presetkeys{todonotes}{size=\tiny}{}

\newcommand\nico[1]{\todo[color=green!40]{#1}}
\newcommand\nicolas[1]{{\todo[inline,color=green!40]{NZ: #1}}}

\newcommand{\I}{{\rm i}} 
\let\nicefrac\frac

\newcommand{\ts}{\hspace{0.1em}}
\newcommand{\tts}{\hspace{0.05em}}

\newcommand{\uptext}[1]{\text{\upshape{#1}}}
\newcommand{\sbinom}[2]{\bigg(\ts\begin{array}{@{}c@{}}#1\\#2\end{array}\ts\bigg)}

\date{\today}

\begin{document}

	\title{Restricted maximum of non-intersecting Brownian bridges}
	\author{Yamit Yalanda and Nicolás Zalduendo}

	\address[Y. Yalanda]{Departamento de Ingeniería Matemática, Universidad de Chile}
	\email{yamit@dim.uchile.cl}
	\address[N. Zalduendo]{Université de Lorraine, CNRS, Inria, IECL, UMR 7502, F-54000 Nancy, France}
	\email{zalduend1@univ-lorraine.fr}

\begin{abstract}
Consider a system of $N$ non-intersecting Brownian bridges in $[0,1]$, and let $\mathcal M_N(p)$ be the maximal height attained by the top path in the interval $[0,p]$, $p\in[0,1]$.
It is known that, under a suitable rescaling, the distribution of $\mathcal M_N(p)$ converges, as $N\to\infty$, to a one-parameter family of distributions interpolating between the Tracy-Widom distributions for the Gaussian Orthogonal and Unitary Ensembles (corresponding, respectively, to $p\to1$ and $p\to0$).
It is also known that, for fixed $N$, $\mathcal M_N(1)$ is distributed as the top eigenvalue of a random matrix drawn from the Laguerre Orthogonal Ensemble. Here we show a version of these results for $\mathcal M_N(p)$ for fixed $N$, showing that $\mathcal M_N(p)/\sqrt{p}$ converges in distribution, as $p\to0$, to the rightmost charge in a generalized Laguerre Unitary Ensemble, which coincides with the top eigenvalue of a random matrix drawn from the Antisymmetric Gaussian Ensemble.
\end{abstract}
	
\maketitle
	
\section{Introduction and main results}

\subsection{Model and motivation}

The model of \emph{non-intersecting Brownian bridges} consists in a collection of $N$ Brownian bridges $(B_1(t) , B_2 (t) , \dotsc , B_N (t))$, all starting from zero at time $t = 0$ and ending at zero at time $t = 1$, and conditioned (in the sense of Doob) to not intersect for $t \in (0, 1)$.
We will always order the $N$ paths increasingly, so that $B_1(t)<B_2(t)<\dotsm<B_N(t)$.
This model and its many variants have attracted a lot of attention in the probability and statistical physics literature, in large part due to their connections with Random Matrix Theory (RMT).
The most basic relation with RMT is that under a simple change of variables, non-intersecting Brownian bridges are mapped to the stationary Dyson Brownian motion for the Gaussian Unitary Ensemble (GUE), which is the evolution of the eigenvalues of an $N\times N$ Gaussian Hermitian matrix whose entries undergo independent (stationary) Ornstein-Uhlenbeck processes (see \eqref{eq:Btolambda}).
In particular, for fixed $t\in (0,1)$ the distribution of $(B_1(t),\dots, B_N(t))$ coincides with the distribution of the GUE eigenvalues, which also means that the properly rescaled top position $B_N(t)$ converges to the Tracy-Widom GUE distribution $F_\text{GUE}$ \cite{tracy1994level} which governs the asymptotic fluctuations of the top eigenvalue of a GUE matrix:
\begin{equation}
\PP\left(\dfrac{B_N(t)}{\sqrt{2t(1-t)}}\leq\sqrt{N}+\tfrac12N^{-1/6}r\right)\xrightarrow[N\to\infty]{}F_\text{GUE}(r).
\end{equation}

Models of non-intersecting random paths also arise in the study of the Kardar-Parisi-Zhang (KPZ) universality class, a broad collection of one-dimensional random growth and related models which share a common fluctuation behavior after suitable space-time rescaling, governed by the \emph{KPZ fixed point} \cite{matetski2021KPZ}.
The limiting distributions and spatial processes arising asymptotically for KPZ models depend on the choice of initial data.
The simplest one corresponds to growth from a single point, referred to as \emph{droplet} or \emph{narrow wedge} initial condition, in which case the spatial fluctuations of the model converge to the \emph{Airy$_2$ process} minus a parabola $\mathcal A_2(x)-x^2$, whose one point marginals are given by the TW-GUE distribution.
The same Airy$_2$ process arises as the scaling limit of the top path of the $N$ non-intersecting Brownian bridges:
\begin{equation}\label{eq:bbAiry2}
  2N^{1/6}\Big(B_N\big(\tfrac12(1+N^{-1/3}x)\big)-\sqrt{N}\Big)\longrightarrow\mathcal A_2(x)-x^2.
\end{equation}
On the other hand, a famous result by Johansson \cite{johansson2003discrete} showed, using an argument based on a connection with KPZ models with \emph{flat} initial data, that $\sup_{x\in\RR}\big(\mathcal A_2(x)-x^2\big)$ is distributed as a Tracy-Widom GOE random variable \cite{tracy1996orthogonal}, corresponding to the asymptotic fluctuations of the top eigenvalue of a Gaussian real symmetric random matrix.
Putting these two facts together one concludes that the rescaled maximal height of $B_N(t)$ converges in distribution to a TW-GOE random variable: more precisely,
\begin{equation}
\PP\Big(2N^{-1/6}\big(\max_{t\in[0,1]}B_N(t)-\sqrt{N}\big)\leq r\Big)\xrightarrow[N\to\infty]{}F_\text{GOE}(4^{1/3}r),\label{eq:maxBN}
\end{equation}
a fact which by now has been proved in several ways in the literature (see for example \cite{nguyen2017non},\cite{fitzgerald2020point}).
We refer the reader to \cite{quastel2014airy} and to the introductions of \cite{nguyen2017extreme,nguyen2017non} and references therein for more background on the facts being discussed here.
The relation between models of non-intersecting Brownian motions and other objects coming from RMT, integrable systems and the KPZ class has been studied intensively from many perspectives, see e.g. \cite{tracy2004differential,tracy2007nonintersecting,warren2007dyson,schehr2008exact,adler2009Dyson,adler2010airy,forrester2011nonintersecting,liechty2012nonintersecting,liechty2017nonintersecting,liechty2020airy}.

Consider now the maximal height of the top path restricted to an interval $[0,p]$, $p\in[0,1]$:
\begin{equation}\label{eq:MNp}
\mathcal{M}_N(p)=\max_{t\in[0,p]}B_N(t).
\end{equation}
Based on the same connection \eqref{eq:bbAiry2} between non-intersecting Brownian bridges and the Airy$_2$ process, together with known KPZ results, the following distributional limit for $\mathcal{M}_N(p)$ can be derived: letting $p(\alpha)=e^{2\alpha}/(1+e^{2\alpha})$ we have, for each $\alpha\in\RR$,
\begin{subequations}\label{eq:F21}
\begin{gather}
	\lim_{N\to\infty} \mathbb P\left(2N^{1/6}\big(\mathcal M_N(p(\alpha N^{-1/3}))\cosh(\min\{\alpha N^{-\nicefrac{1}{3}},0\})-\sqrt{N}\big)\leq r\right) = F^{(\alpha)}_{2\to1} (r)\label{eq:F21a}\\
	\text{with}\qquad F^{(\alpha)}_{2\to1} (r)\xrightarrow[\alpha\to\infty]{}F_\text{GOE}(4^{1/3}r),
	\qquad F^{(\alpha)}_{2\to1} (r)\xrightarrow[\alpha\to-\infty]{}F_\text{GUE}(r).\label{eq:F21b}
\end{gather}
\end{subequations}
More concretely, $F^{(\alpha)}_{2\to1}$ corresponds to the marginal distribution of the Airy$_{2\to1}$ process arising as the scaling limit of KPZ models with \emph{half-flat} initial data \cite{borodin2008transition}, whose distribution is known to interpolate between TW-GOE and TW-GUE.
The $\alpha\to\infty$ limit to TW-GOE corresponds to \eqref{eq:maxBN}.
For the other case, $\alpha\to-\infty$ corresponds to the maximum in \eqref{eq:MNp} being computed in an interval $[0,p]$ with $p\to0$, while for large $N$ the top path is known to concentrate around the curve $2\sqrt{Nt(1-t)}$, and thus to first order one expects the maximum on $[0,p]$ to occur near $p$ as $p\to0$ and to be close to $2\sqrt{Np(1-p)}$; the hyperbolic cosine factor in \eqref{eq:F21a} compensates for this decay, and the fluctuations then come from the value of $B_N$ at the right edge of the interval, which has TW-GUE fluctuations.
The derivation of \eqref{eq:F21a} is by now relatively standard.
There are two possible routes.
The first one is to use a version of \eqref{eq:bbAiry2} where the convergence holds uniformly in compact sets \cite{corwin2014gibbs} together with estimates on the tails of the Airy$_2$ process to show that the limit equals $\mathbb P(\sup_{x\leq\alpha}(\mathcal A_2(x)-x^2)+\min\{\alpha,0\}^2\leq r)$, which was shown in \cite{quastel2013supremum} to be given by the marginals of the Airy$_{2\to1}$ process.
The other one involves calculating the left hand side directly and checking that it converges to the known expressions for $F^{(\alpha)}_{2\to1}$ (this can be done using the formula derived Proposition below \ref{prop:maximum_height} below, but we omit the details).

The focus of this paper is the random variable $\mathcal M_N(p)$ for fixed $N\in\NN$.
In \cite{nguyen2017non} it was shown that in the case $p=1$ (i.e. when we look at the maximal height of the system of non-intersecting Brownian bridges over the whole interval $[0,1]$), the square of $\mathcal M_N(1)$ is distributed as the largest eigenvalue of a random matrix drawn from the Laguerre Orthogonal Ensemble, which provides a \emph{finite $N$} version of \eqref{eq:maxBN}. 
We will be interested in the other limiting case, i.e. the distribution of $\mathcal{M}_N(p)$ (after proper rescaling) as $p\to0$.
Some results available in the literature strongly suggest a connection between this distribution and that of the largest eigenvalue of a different random matrix family (see Section 1.3 for more details).
The aim of this paper is to prove this connection and identify this limiting distribution which, as we prove, corresponds to the largest eigenvalue of a matrix of the Gaussian Antisymmetric Ensemble. To arrive at this result, we will exploit the relation between this family of matrices and the Generalized Laguerre Unitary Ensemble (see Theorem 1 below).

\subsection{Main result}

Let $X$ be an $N\times m$ matrix, $m\geq N$, with independent real or complex standard Gaussian entries (in the complex case, the real and imaginary parts of each entry are independent with variance $1/2$).
The $N\times N$ matrix $M=XX^*$ is sometimes referred to as a \emph{real} or \emph{complex Wishart matrix}, and plays a central role in multivariate statistics as the sample covariance matrix of a Gaussian population.
The eigenvalues $\lambda_1\leq\dotsm\leq\lambda_N$, and particularly the largest eigenvalues of Wishart matrices, are of particular importance in statistical applications such as principal component analysis.
The joint distribution of this eigenvalues can be computed explicitly, and is given by (see \cite{forrester2010log}, page 91)
\begin{equation}\label{eq:laguerre-ensemble}
	\frac{1}{Z_N}\prod_{1\leq i<j\leq N} |\lambda_i-\lambda_j|^\beta\prod_{i=1}^N
	\lambda_i^{\nicefrac{\beta a}{2}}e^{-\beta\nicefrac{\lambda_i}{2}}
\end{equation}
with $\beta=1$ in the real case, $\beta=2$ in the complex case and $a=m-N+1-\nicefrac{2}{\beta}$ ($Z_N$ is a normalization constant).
The weights $\lambda_i^{\nicefrac{\beta a}{2}}e^{-\beta\nicefrac{\lambda_i}{2}}$ are associated to the \emph{generalized Laguerre orthogonal polynomials}, and based on this the random matrix $M$ is said to belong to the \emph{Laguerre Orthogonal Ensemble} in the real case and to the \emph{Laguerre Unitary Ensemble} in the complex case (the orthogonal and unitary names coming from the fact that the distribution of $M$ is invariant under conjugation by fixed matrices from, respectively, the groups $O(N)$ and $U(N)$).

The distribution of the eigenvalues of an $N\times N$ LOE or LUE matrix depends on the parameter $a$, which is determined by the aspect ratio of the matrix $X$.
Let $\bar\lambda_{\text{L(O/U)E},a}$ denote the largest eigenvalue of such a matrix and let $F_{\text{L(O/U)E},a}$ denote its distribution, i.e.
\begin{equation}
F^{(a)}_{\text{L(O/U)E},N}(r)=\PP\big(\bar\lambda_{\text{L(O/U)E},a}\leq r\big).\label{eq:FaLag}
\end{equation}
Then \cite{nguyen2017non} showed that
\begin{equation}
\PP\!\left(\mathcal M_N(1) \leq r\right) = F^{(1)}_{\text{LOE},N}(4r^2).\label{eq:nibm-loe}
\end{equation}
In other words, $4\max_{t\in[0,1]}B_N(t)^2$ is distributed as the largest eigenvalue 
of an LOE matrix $M=XX^T$ with $X$ of size $N\times(N+1)$.

The largest eigenvalue of an LOE matrix is known to converge to a TW-GOE random variable \cite{johnstone2001distribution} (more precisely, in our case we have $\lim_{N\to\infty}F^{(1)}_{\text{LOE},N}(4N+4^{2/3}N)=F_{\text{GOE}}(r)$), and hence \eqref{eq:nibm-loe} can indeed be regarded as a finite $N$ analog of the first limit in \eqref{eq:F21b}.
It is natural then to wonder about a version of the second limit in \eqref{eq:F21b}, corresponding to studying $\mathcal M_N(p)$ with $p\to0$.

Of course, as $p\to0$ the maximum of $B_N(t)$ for $t\in[0,p]$ goes to $0$ too, so in order to see something interesting we need to rescale $\mathcal M_N(p)$ before taking the limit.
$B_N(t)$ grows like $\sqrt{t}$ for small $t$, so we will study the distributional limit
\begin{equation}\label{eq:hatM0}
\mathcal{\hat M}_N(0)\coloneqq\lim_{p\to0}\frac{\mathcal M_N(p)}{\sqrt{p}}.
\end{equation}
In view of \eqref{eq:F21} and \eqref{eq:nibm-loe}, a natural guess is that this limit be related with the Laguerre Unitary Ensemble.
Our main result, which we state next, confirms this, with the caveat that the parameter $a$ has to be allowed to take non-integer values.
Note that \eqref{eq:laguerre-ensemble} defines a probablity density for the vector $(\lambda_1,\dotsc,\lambda_N)\in\RR^N$ for any $a>-1$.
In general the $\lambda_i$'s do not arise as eigenvalues of a naturally defined random matrix, and we think of \eqref{eq:laguerre-ensemble} as defining a point process, or a \emph{Coulomb gas}, in the real line.
In any case, $F^{(a)}_{\text{L(O/U)E},N}$ still makes sense as the distribution of the largest charge $\bar\lambda_{\text{L(O/U)E},a}$ among the $\lambda_i$'s.

\begin{theorem}\label{thm:maximum_height} For every $N\in \mathbb N$ and $r\geq 0$,
\begin{equation}\label{eq:maximum_height}
\lim_{p\to0}\PP\big(\mathcal M_N(p)\leq\sqrt{p}\ts r)
=\begin{dcases*}
F^{(1/2)}_{\uptext{LUE},N/2}(r^2/2) & if $N$ is even,\\
F^{(-1/2)}_{\uptext{LUE},(N+1)/2}(r^2/2)& if $N$ is odd.
\end{dcases*}
\end{equation}
In other words, the distributional limit $\mathcal{\hat M}_N(0)$ in \eqref{eq:hatM0} is well defined, and  $\mathcal{\hat M}_N(0)^2/2$ has the distribution of the largest charge 
in the \emph{generalized Laguerre Unitary Ensemble} defined through \eqref{eq:laguerre-ensemble} with size $\lfloor(N+1)/2\rfloor$, $\beta=2$ and $a=\frac12(-1)^{N}$.
\end{theorem}

Together, the identities \eqref{eq:nibm-loe} and \eqref{eq:maximum_height} provide an analog of \eqref{eq:F21}, which can be stated as follows:
\begin{equation}\label{eq:forLtoG}
\begin{gathered}
\hat F^p_N(r)\coloneqq\PP\big(\mathcal M_N(p)\leq\sqrt{p}\ts r)
\longrightarrow\begin{dcases*}
F_{\text{LOE},N}^{(1)}(4r^2) & as $p\to1$,\\
F_{\text{LUE},\lfloor(N+1)/2\rfloor}^{(\frac12(-1)^N)}(r^2/2) & as $p\to0$.
\end{dcases*}
\end{gathered}
\end{equation}
From results of \cite{johansson2000shape} and \cite{johnstone2001distribution} it is known that, under the scaling implied by \eqref{eq:F21}, the LOE and LUE distributions on the right hand side of \eqref{eq:forLtoG} converge to TW-GOE and TW-GUE, so these limits are consistent with \eqref{eq:F21b} (see the discussion after Theorem 1.2 in \cite{nguyen2017non} for more details).

\subsection{Matrix model for \texorpdfstring{$\mathcal{\hat M}_0$}{M0}}

We have obtained the distributions $F_{\text{LUE},N}^{(\pm1/2)}$ appearing in Theorem \ref{thm:maximum_height} as those of the largest charge in a generalized Laguerre Unitary Ensemble.
This ensemble appears not to be related to some sort generalized Wishart matrix but, remarkably, there is another simple random matrix model whose eigenvalues recover this distribution.

Given $n\in \NN$, a $n\times n$ (purely imaginary) random matrix $H$ is said to be drawn from the \emph{Antisymmetric Gaussian Ensemble} if $H=\frac{\I}{2}(X-X^{\sf T})$ for $X$ a real random matrix with independent standard Gaussian entries.
The non-zero eigenvalues of this matrix come in pairs $\{\pm \lambda_j\}_{j=1,\dots,\lfloor n/2\rfloor}$, and the joint density function for the positive eigenvalues is given by \cite{mehta2004random,dumitriu2010tridiagonal} 
\[\frac{1}{Z_n}
\prod_{i=1}^{\lfloor n/2\rfloor}
\lambda_i^{1+(-1)^{n+1}}e^{-\lambda_i^2}
\prod_{1\leq j<k\leq \lfloor n/2\rfloor} 
(\lambda_i^2-\lambda_j^2)^2,
\]
where $Z_n$ is a normalization constant.
Applying the change of variables $\lambda_j^2\longmapsto\lambda_j$ this joint density becomes that of the generalized LUE with $N=n-1$ and $a=\frac12(-1)^{n}$.

In other words, this shows that
\[\hat M_N(0)\stackrel{\uptext{dist}}{=}\sqrt{2} \hspace{1mm} \bar\lambda_{\text{AntiGE},N+1},\]
where $\bar\lambda_{\text{AntiGE},n}$ denotes the largest eigenvalue of an $n\times n$ matrix from the Antisymmetric Gaussian Ensemble.

Let us briefly explain the connection between this identity and some results which are available in the literature.
Our system of non-intersecting Brownian bridges can be translated into GUE Dyson Brownian motion, see Proposition 2.
In \cite{borodin2009maximum} (Theorem 1) the authors established an identity in distribution for the maximum of the top path in this process, showing that it can be expressed as the location of the top path in a system of non-colliding paths with a wall.
This connection does not yield exactly the distribution for $\mathcal{M}_N(p)$, but using a comparison based on Brownian scaling one can obtain from that random variable information about $\hat{\mathcal{M}}_N(0)$. Systems of non-colliding paths with a wall have also arisen in other contexts in the KPZ universality class, with a connection suggested to the Gaussian Antisymmetric Ensemble (see e.g. \cite{tracy2007nonintersecting}, \cite{borodin2009twospeed}).
However, the connection with this random matrix family or the Generalized Laguerre Unitary Ensemble in our context appears not to have been proved (or in fact, made explicit) in any form which is useful for the problem we tackle.
Our contribution in this paper is to establish this precisely, and to do so based on a direct proof.


\section{Fredholm determinant formula for the distribution of the restricted maximal height}
\label{sec:proof_N_paths}

The goal of this section is to derive Fredholm determinant formulas for the distribution of the random variables $\mathcal{\hat M}_N(0)$ and $\mathcal M_N(p)$.
Let $\varphi_n$ be the \emph{harmonic oscillator functions}, or \emph{Hermite functions}, defined by $\varphi_n(x)=e^{-x^2/2}p_n(x)$, with $p_n$ the $n$-th Hermite polynomial (see e.g. \cite[\S 18.3]{NIST:DLMF}) normalized so that $\|\varphi_n\|_2=1$, and then define the \emph{Hermite kernel} as
\begin{equation}\label{eq:defKN}
  K_\text{H}^N(x,y)=\sum_{n=0}^{N-1}\varphi_n(x)\varphi_n(y).
\end{equation}
We also define the \emph{reflection operator} $\varrho_r$ as well as projection and multiplication operators $\chi_r$ and $E_r$, acting on $L^2(\RR)$, through
\begin{equation}\label{eq:def-varrho-M}
  \varrho_rf(x)=f(2r-x),\qquad\chi_rf(x)=\mathds{1}_{x>r}f(x),\qquad\text{and}\qquad E_rf(x)=e^{rx}f(x).
\end{equation}
Let also $\bar\chi_r=I-\chi_r$.

\begin{prop}\label{prop:maximum_height}
	Consider $N\in \mathbb N$ and $r\geq 0$. Fix $p\in(0,1)$ and let $\alpha =\frac{1}{2}\log(\frac{p}{1-p})$.
	Then,
	\begin{multline}
	\label{eq:det_formula}
	\PP(\sqrt{2}\mathcal M_N(p) \leq r) \\
	= \det\!\left(I-K_\text{H}^N\chi_{r\cosh(\alpha)}K_\text{H}^N - K_\text{H}^NE_{r\sinh(\alpha)}\varrho_{r\cosh(\alpha)}E_{-r\sinh(\alpha)}\bar\chi_{r\cosh(\alpha)}K_\text{H}^N\right),
	\end{multline}
	where $\det$ means the Fredholm determinant on the Hilbert space $L^2(\RR)$.
\end{prop}

For the definition and properties of the Fredholm determinant we refer the reader to \cite{simon2005trace} or \cite[Section 2]{quastel2014airy}.

\begin{proof}
It turns out to be convenient to express $\mathcal M_N(p)$ in terms of the stationary (GUE) Dyson Brownian motion, for which formulas are cleaner.
This is the process which describes the evolution of the eigenvalues of an $N\times N$ matrix whose entries evolve as independent stationary Ornstein-Uhlenbeck processes $dX_t=-X_tdt+\sigma dB_t$ where $B$ is a standard Brownian motion and $\sigma=1/\sqrt{2}$ for off-diagonal entries, $\sigma=1$ on the diagonal.
For each fixed $t$, the ordered eigenvalues $\lambda_1(t)\leq\dotsm\leq\lambda_N(t)$ are distributed, up to scaling, as the eigenvalues of an $N\times N$ GUE matrix, and if we let this stationary process be defined for $t\in\RR$, then one has (see e.g. \cite{tracy2007nonintersecting})
\begin{equation}
(B_i(s))_{i=1,\dots,N} \stackrel{\uptext{dist}}{=} \Big(\sqrt{2s(1-s)}\lambda_i\big(\tfrac{1}{2}\log\! \big(\tfrac{s}{1-s}\big)\big)\Big)_{i=1,\dots N}\label{eq:Btolambda}
\end{equation}
as processes defined for $s\in [0,1]$.
Thus, applying the change of variables $s\to \frac{e^{2t}}{e^{2t}+1}$ we have
$\sqrt{2}\mathcal M_N(p)= \sup_{t\in (-\infty, \alpha]}\frac{\lambda_N(t)}{\cosh (t)}$
with $\alpha$ as in the statement of the result.
Hence 
\begin{equation}
\PP(\sqrt{2}\mathcal M_N(p)\leq r)=\PP(\lambda_N(t)\leq r\cosh(t)~\forall\tts t\leq\alpha).\label{eq:PPMNlambdaN}
\end{equation}
We will focus now on the random variable on the right hand side.

Let $D$ denote the operator $Df(x)=-\frac12(f''(x)-(x^2-1)f(x))$.
The Hermite functions $\varphi_n$ are eigenfunctions for $D$, $D\varphi_n=n \varphi_n$ for $n\geq0$, so that that $K_\text{H}^N$ is the projection operator onto the space span$\{\varphi_0,\dotsc,\varphi_{N-1}\}$ associated to the first $N$ eigenvalues of $D$.
In particular, $e^{tD}K_\text{H}^N$ is well defined for all $t$, with integral kernel is given by 
\begin{equation}\label{eq:etDKN}
e^{t{\sf D}}K_\text{H}^N(x,y)=\sum_{n=0}^{N-1}e^{tn}\varphi_n(x) \varphi_n(y).\end{equation}
It was shown in \cite{borodin2015multiplicative} that
\begin{equation}\label{eq:cont_stat_sdbm}
\begin{aligned}
\mathbb P\left(\lambda_N(t)\leq r\cosh(t)~\forall\tts t\in [-L,\alpha]\right)&= \det (I - K_\text{H}^N +\Theta_{L,\alpha}^r e^{(\alpha+L)D}K_\text{H}^N)\\
&= \det (I - K_\text{H}^N +e^{(\alpha+L)D}K_\text{H}^N\Theta_{L,\alpha}^r K_\text{H}^N)
\end{aligned}
\end{equation}
where $\Theta_{L,\alpha}^r$ is the solution operator for a certain boundary value problem involving $D$ (see \cite[Prop. 2.1]{nguyen2017non}), and where in the second equality we used the facts that $e^{sD}K_\text{H}^N=K_\text{H}^Ne^{sD}K_\text{H}^N$ and $K_\text{H}^N=e^{sD}K_\text{H}^Ne^{-sD}K_\text{H}^N$ together with the cyclic property of the determinant.
We use now the decomposition for $\Theta_{L,\alpha}^r$ given in \cite{nguyen2017non}: 
\[\Theta_{L,\alpha}^r = e^{-(\alpha+L)D}-e^{-(\alpha+L)D}\chi_{r\cosh (\alpha)} - R^r_{L,\alpha}\bar\chi_{r\cosh (\alpha)} - \Omega_{L,\alpha}\]
with $\Omega_{L,\alpha} = \chi_{r\cosh(L)}\left(e^{-(\alpha+L)D}-R_{L,\alpha}^r\right)\bar\chi_{r\cosh(\alpha)}$ and $R_{L,\alpha}^r$ an operator defined on $L^2(\RR)$ through its kernel defined for $x,y\in \RR$ (after some rearranging) as
\[R_{L,\alpha}^r(x,y)=(\pi(e^{2\alpha}-e^{-2L}))^{-1/2}e^{\alpha-u_1x^2+u_2x+u_3},\]
with
\[u_1=\tfrac{1+e^{2\alpha+2L}}{2(e^{2\alpha+2L}-1)},\:u_2=\tfrac{2e^{L}(r+e^{2\alpha}r-e^{a}y)}{e^{2\alpha+2L}-1},\:u_3=-\tfrac{2(1+e^{2\alpha})(1+e^{2L})r^2-4e^{\alpha}(1+e^{2L})ry+(1+e^{2\alpha+2L})y^2}{2(e^{2\alpha+2L}-1)}.\]
Using this, the determinant in \eqref{eq:cont_stat_sdbm} is equal to 
\[\det \left( I - K_\text{H}^N \chi_{r\cosh (\alpha)} K_\text{H}^N - e^{(\alpha+L)D}K_\text{H}^NR_{L,\alpha}^r\bar\chi_{r\cosh (\alpha)}K_\text{H}^N - e^{(\alpha +L)D}K_\text{H}^N\Omega_{L,\alpha}K_\text{H}^N\right).\]
$\Omega_{L,\alpha}$ is to be thought of as an error term, and in fact essentially the same proof as that of \cite[Lem. 2.3]{nguyen2017non}, shows that $e^{(\alpha+L)D}K_\text{H}^N\Omega_{L,\alpha}K_\text{H}^N\xrightarrow[L\to\infty]{}0$ in trace norm.
The Fredholm determinant is continuous with respect to the trace norm, so this together with \eqref{eq:PPMNlambdaN} and \eqref{eq:cont_stat_sdbm} shows that $\PP(\sqrt{2}\mathcal M_N(p)\leq r)=\lim_{L\to\infty}\mathbb P\left(\lambda_N(t)\leq r\cosh(t)~\forall\tts t\in [-L,\alpha]\right)$ equals
\[\det\!\left( I - K_\text{H}^N \chi_{r\cosh (\alpha)} K_\text{H}^N - e^{(\alpha+L)D}K_\text{H}^NR_{L,\alpha}^r\bar\chi_{r\cosh (\alpha)}K_\text{H}^N \right).\]
We will see now that the last kernel inside this last determinant does not depend on $L$, and equals the one appearing in the statement of the proposition.
%
%
%
The contour integral representation of the Hermite function $\varphi_n(x)=(2^nn!\sqrt{\pi})^{-\frac{1}{2}}e^{-x^2/2}\frac{n!}{2\pi\I}\oint dw \frac{e^{2wx-w^2}}{w^{n+1}}$ (the integral is over a simple contour around the origin) together with \eqref{eq:etDKN} imply that $e^{(\alpha+L)D}K_\text{H}^NR_{L,\alpha}^r(x,y) = \int_{\RR}dz \sum_{n=0}^{N-1}e^{(\alpha+L)n}\varphi_n(x)\varphi_n(z)R_{L,\alpha}^{r}(z,y)=\frac{n!}{2\pi \I}\oint dw \,\frac{e^{-w^2}}{w^{n+1}}\int_{\RR}dz\, e^{-z^2/2+2w z-u_1z^2+u_2z+u_3}$.
The $z$ integral is a Gaussian integral, and computing it, changing variables $t\longmapsto te^{(\alpha+L)}$ and regrouping terms in the exponent we get that $e^{(\alpha+L)D}K_\text{H}^NR_{L,\alpha}^r(x,y)$ equals 
\[\sum_{n=0}^{N-1} \varphi_n(x) \big[(2^nn!\sqrt{\pi})^{-\frac{1}{2}}e^{-(e^{-\alpha}r+e^{\alpha}r-y)^2/2}\frac{n!}{2\pi i} \oint dt \frac{e^{-t^2 +2t(e^{-\alpha}r+e^{\alpha}r-y)}}{t^{n+1}} \big] e^{r^2 \sinh(2\alpha)-2r\sinh(\alpha)y}.\]
The term inside the brackets is just $\varphi_n(e^{-\alpha}r+e^{\alpha}r-y)=\varphi_n(2r\cosh(\alpha)-y)$ and therefore $e^{(\alpha+L)D}K_\text{H}^NR_{L,\alpha}^r=K_\text{H}^Ne^{r\sinh(\alpha)\xi}\varrho_{r\cosh(\alpha)}e^{-r\sinh(\alpha)\xi}$, whence the result follows.
\end{proof}

\begin{corollary}\label{cor:maximum_height_0}
	\begin{equation}
	\label{eq:det_formula_0}
	\lim_{p\to0}\PP(\mathcal{M}_N(p) \leq \sqrt{2p}\tts r) = \det\!\left(I-K_\text{H}^N\chi_{r}K_\text{H}^N - K_\text{H}^NE_{-r}\varrho_{r}E_{r}\bar\chi_{r}K_\text{H}^N\right).
	\end{equation}
\end{corollary}

\begin{proof}
We need to compute the limit of the right hand side of \eqref{eq:det_formula} with $r$ replaced by $2\sqrt{p}\tts r$.
Since for $\alpha=\frac12\log(\frac{p}{1-p})$ one has $2\sqrt{p}\cosh(\alpha)\longrightarrow1$ and $2\sqrt{p}\sinh(\alpha)\longrightarrow-1$, it is easy to see that the kernel inside the resulting Fredholm determinant converges pointwise to the kernel appearing on the right hand side of \eqref{eq:det_formula_0}.
In order to upgrade this to convergence of the Fredholm determinant itself we need to show that the convergence holds in trace norm.
The proof of this is standard, and can be done by adapting the arguments used 	in \cite[Appdx. B]{nguyen2017non}; the present setting is simpler, and we omit the details.
\end{proof}

\section{Proof of Theorem \ref{thm:maximum_height}}

Recall the definition of the Laguerre Unitary Ensemble in \eqref{eq:laguerre-ensemble} (with $\beta=2$).
By standard methods in RMT and determinantal point processes (see e.g. \cite{johansson2005random,forrester2010log}), the distribution $F_{\text{LUE,m}}^{(a)}$ of the righmost charge $\lambda_m$ in the size $m$ generalized LUE can be expressed as a Fredholm determinant:
\begin{equation}\label{eq:CDFma}
F_{\text{LUE},m}^{(a)}(r)=\det(I-K_\text{L}^{m,(a)}\chi_{r}K_\text{L}^{m,(a)}),
\end{equation}
where the \emph{Laguerre kernel} $K_\text{L}^{m,(a)}$ is defined as
\[
K_\text{L}^{m,(a)}=\sum_{k=0}^{m-1}\psi_k^{(a)}(x)\psi_k^{(a)}(y),
\]
for $\psi_k^{(a)}$ the \emph{generalized Laguerre functions} which are defined as $\psi_k^{(a)}(x)=x^{a/2}e^{-x/2}L_k^{(a)}(x)\mathds{1}_{x\geq0}$, with $L^{(a)}_k$ the degree $k$ generalized Laguerre polynomial satisfying $\int_0^{\infty}dx\,L_m^{(a)}(x)L_n^{(a)}(x)x^{a}e^{-x}=\frac{\Gamma(n+a+1)}{n!}\delta_{mn}$, i.e. the orthogonal polynomials with respect to the weight $x^ae^{-x}$,  normalized so that $\|\psi_k\|_2=1$.

Let
\begin{equation}
U_N(r)=\det\!\left(I-K_\text{H}^N\chi_{r}K_\text{H}^N - K_\text{H}^NE_{-r}\varrho_{r}E_{r}\bar\chi_{r}K_\text{H}^N\right),\label{eq:defUN}
\end{equation}
which is the right hand side of \eqref{eq:det_formula_0}.
Our goal is to prove:

\begin{prop}\label{prop:CDF}
For every $N\in \mathbb N$ and $r\geq0$,
\begin{equation}
U_N(r)=\det\!\big(I-K_\text{L}^{\lfloor(N+1)/2\rfloor,\frac12(-1)^N}\chi_{r^2}\ts K_\text{L}^{\lfloor(N+1)/2\rfloor,\frac12(-1)^N}\big).\label{eq:CDF}
\end{equation}
\end{prop}

Theorem \ref{thm:maximum_height} follows directly from this and Corollary \ref{cor:maximum_height_0} after replacing $r$ by $\sqrt{2}\tts r$.
We turn now to the proof of the proposition.

The kernel $K_\text{H}^N$ has finite rank, so the above Fredholm determinant can be expressed as the determinant of a finite matrix.
In order to do so, let $J_1:\ell^2(\{0,\dots,N-1\})\to L^2(\mathbb{R})$ and $J_2:L^2(\mathbb{R})\to \ell^2(\{0,\dots,N-1\})$ be the kernels given by
\[J_1(x,n)=\varphi_n(x),\qquad\text{and}\qquad J_2(n,y)=\varphi_n(y),\]
where $\varphi_n(x)=C_ne^{-x^2/2}H_n(x)$ are the Hermite functions introduced above.
Here $H_n$ are the standard Hermite polynomials and $C_n=(\sqrt{\pi}2^nn!)^{-\nicefrac{1}{2}}$.
We have $J_1J_2=K_\text{H}^N$ while $J_2J_1$ equals the identity in $\ell^2(\{0,\dots,N-1\})$, so applying the cyclic property of the Fredholm determinant to the right hand side of \eqref{eq:defUN} we deduce that $U_N(r)=\det(I-J_2P_rJ_1-J_2M_{-r}\varrho_r M_r\bar\chi_rJ_1)$, which is now a Fredholm determinant on $\ell^2(\{0,\dots,N-1\})$.
As such, it can be rewritten as the determinant of an $N\times N$ matrix as desired:
\[U_N(r)=\det(I-M)\]
for $M$ the matrix indexed by $\{0,\dots N-1\}$ and given by
\begin{align*}
M_{j,k}&=\big(J_2\chi_rJ_1+J_2M_{-r}\varrho_rM_r\bar\chi_rJ_1\big)(j,k)
=\int_{r}^{\infty}dz\, \varphi_j(z) \! \left[\varphi_k(z)+\varphi_k(2r-z) e^{2r^2-2rz}\right]\\
&=C_jC_{k}\int_{0}^{\infty}dz\,H_j(r+z)\! \left[H_{k}(r+z)+H_{k}(r-z)\right]e^{-(r+z)^2}\\
&=C_jC_{k}\left[\sum_{l=0}^{k}\sbinom{k}{l}(4r)^{k-l}(-1)^l \int_{r}^{\infty}dz\,H_j(z) H_{l}(z)e^{-z^2}+\int_{r}^{\infty}dz\,H_j(z) H_{k}(z)e^{-z^2}\right],
\end{align*}
where in the second line we used the change of variable $z\longmapsto z+r$ and in the third one the identity 
$H_n(x+y)=\sum_{k=0}^{n}\binom{n}{k}H_k(x)(2y)^{n-k}$ \cite{huang2000introduction}.
If we now define the matrices
\begin{equation}
\label{eq:matrix_F}
F_{j,k}=\sbinom{k}{j}(4r)^{k-j}(-1)^j\mathds{1}_{j\leq k}+\delta_{j,k},
\qquad Q_{j,k}=C_j^2\int_{r}^{\infty}dz\,H_j(z) H_{k}(z)e^{-z^2},
\end{equation}
then using the cyclic property of the determinant we obtain 
\begin{equation}
U_N(r)=\det(I-QF).
\end{equation}

Now let $G_N(r)$ denote the right hand side of \eqref{eq:CDF}, so that our goal is to prove $U_N=G_N$. In everything that follows we let $b=1$ if $N$ is odd and $b=2$ if $N$ is even. The Laguerre kernel appearing inside the Fredholm determinant in \eqref{eq:CDF} can be expressed in terms of Laguerre polynomials using the relation (see \cite{andrews1999special})
\[
L_n^{(\frac{1}{2}(-1)^b)}(u)=
\frac{(-1)^n2^{-2n-b+1}}{n!} u^{-(1+(-1)^b)/4} H_{2n+b-1}(u^{1/2}),
\]
computing as we just did above using the cyclic property of the determinant leads to
\[G_N(r)=\det(I-A),\]
where the matrix $A$ is indexed by $\{0,\dots, \lfloor \frac{N-1}{2}\rfloor\}$ and is given by
$
A_{j,k}=2Q_{2j+b-1,2k+b-1}.
$
Note that the matrices $QF$ and $A$ appearing within the determinants for $U_N$ and $G_N$ have different dimensions.
In principle, this makes it hard to compare the two determinants.
The key is that $M$, which is of size $N$, actually has rank $\lfloor\frac{N+1}{2}\rfloor$, which is the size of $A$.
In order to see, and use this, we first introduce the matrices 
\begin{equation}\label{definicion de S}
	S_{j,t}=
	\frac{t!}{2^{2j}(t-N+2j+b)!}
	H_{t-N+2j+b}(r),\quad j=0,\dotsc,\tfrac12(N-b),\, t=0,\dotsc,N-1,
\end{equation}
\begin{equation}\label{definicion de T}
T_{u,k}=
\frac{(-1)^{\frac{N-b}{2}+k}(-\I)^{u}\ts 2^{2k} }{u!(N-2k-b-u)!}
H_{N-2k-b-u}(\I r),\quad u=0,\dotsc, N-1,\, k=0,\dotsc,\tfrac12(N-b).
\end{equation}
Here, and in everything that follows, we use the convention 
\[H_m(x)=0\quad\text{and}\quad\frac{1}{m!}=0\qquad\forall\,m<0.\]
We note that $H_n(x)$ has the same parity as $n$, and thus $\I^nH_n(\I r)$ is real, which shows that $T$ is a real matrix.
Proposition \ref{prop:CDF} will follow from the following two results:

\begin{prop}\label{propiedades_de_T_y_S}
The matrices $S$ and $T$ satisfy:
\begin{enumerate}[label=\uptext{(\arabic*)}]
    \item $2TS=F$ (with $F$ as defined in \eqref{eq:matrix_F}).
    \item $ST=I$, the identity matrix of size $\lfloor(N+1)/2\rfloor$.
\end{enumerate}
\end{prop}

Property (1) implies that $F$, and thus $QF$, have rank $\lfloor(N+1)/2\rfloor$ as claimed.

\begin{prop}\label{propiedad de 2SC^2QT=A}
$2SQT=A$.
\end{prop}

\begin{proof}[Proof of Proposition \ref{prop:CDF}]
\eqref{eq:CDF} is now a simple consequence of the above identities and the cyclic property of the determinant:
\[G_N(r)=\det(I-A)=\det(I-2SQT)=\det(I-2QTS)=\det(I-QF)=U_N(r).\qedhere\]
\end{proof}

We turn to the proofs of the two propositions.
They depend on the following two lemmas, whose proofs are deferred to the end of the section.

\begin{lem}\label{lem:first lemma} 
Let $m$ be a non-negative integer.
\begin{enumerate}[label=\uptext{(\arabic*)}]
\item\label{binomial hermite 1} 
If $m>0$ then
\begin{gather}
\label{eq:suma impar}
\quad\sum_{s\geq0}
\sbinom{m}{2s+1}
(-1)^s
H_{m-2s-1}(\I r)
H_{2s+1}(r)
=
\tfrac{1}{2}\I^{m-1}
(4r)^{m},\\
\label{eq:suma par}
\sum_{s\geq0}
\sbinom{m}{2s}
(-1)^s
H_{m-2s}(\I r)
H_{2s}(r)
=
\tfrac{1}{2}\I^{m}
(4r)^{m}.
\end{gather}
\item \label{binomial hermite 2}		
$\displaystyle
\sum_{t=0}^{m}
(-\I)^{t}
\sbinom{m}{t}
H_{t}(r)
H_{m-t}(\I r)
=
\delta_{m,0}.
$
\item \label{binomial hermite 3} For all integers $d$ and $m\geq0$ such that $d+m\geq0$,
\[
\sum_{u=0}^{m}
\sbinom{m}{u}
(-\I)^{u}
\,H_{m-u}(\I r)
H_{u+m+d}(r)
=\frac{\I^m2^m(m+d)!}{d!}H_d(r).
\]
\end{enumerate}
\end{lem}

\begin{lem}\label{lem:second lemma} For all integers $m\geq0$, $t\geq 0$ and $d\geq m$:
\begin{enumerate}[label=\uptext{(\roman*)}]
\item \label{binomial 1} 
If $t\geq m+1$ then $\displaystyle
\sum_{l\geq 0}(-1)^l 
\sbinom{t}{l}
\sbinom{t-l-1}{m-l}=
(-1)^{m} .
$
\item\label{binomial 2} $\displaystyle
\sum_{u=0}^{k}(-1)^u\frac{(t+u)!}{u!(k-u)!(t+u-h)!}
=
(-1)^k\frac{h!}{k!}
\sbinom{t}{h-k}.
$
\end{enumerate}
\end{lem}

\begin{proof}[Proof of Proposition~\ref{propiedades_de_T_y_S}]
We have
\[(2TS)_{j,k}=2(-1)^{\frac{N-b}{2}}(-\I)^j \dfrac{k!}{j!}\sum_{s=0}^{\frac{N-b}{2}}(-1)^s \dfrac{H_{N-2s-b-j}(\I r)}{(N-2s-b-j)!}\dfrac{H_{k-N+2s+b}(r)}{(k-N+2s+b)!}.\]
Note that the right hand side vanishes whenever $N-2s-b-j<0$ or $k-N+2s+b<0$, and thus for $j>k$ we have $(2TS)_{j,k}=0=F_{j,k}$, which proves (1) in this case.
Now take $j\leq k$.
Using the same argument we can replace the summation over $0\leq s\leq\frac{N-b}{2}$ by a summation over $s\geq\frac{N-b}{2}-\ell$ with $\ell = \lfloor \nicefrac{k}{2}\rfloor$ and then change variables to rewrite the above expression as
\[(2TS)_{j,k}=2(-1)^{-\ell}(-\I)^j\sbinom{k}{j}\sum_{s=0}^{\infty}(-1)^s \sbinom{k-j}{k-2\ell +2s}H_{2\ell -j -2s}(\I r)H_{k-2\ell+2s}(r)\]
In the case $k=j$, when $k$ is even the only term surviving from the sum is $s=0$ and the sum equals $1$ (using $H_0\equiv1$), while if $k$ is odd all terms in the sum vanish, so using this we get $(2TS)_{k,k} = (-1)^{k}+1= F_{k,k}$.
Assume now that $j<k$. If $k$ es even $(2TS)_{j,k}$ equals
\begin{align*} 
2(-1)^{\frac{k}{2}}(-\I) \sbinom{k}{j}\sum_{s=0}^{\infty} (-1)^s \sbinom{k-j}{2s}H_{k-j-2s}(\I r)H_{2s}(r)
= 2(-1)^{\frac{k}{2}}(-\I)^j\sbinom{k}{j}\tfrac{1}{2}\I^{k-j}(4r)^{k-j}
\end{align*}
by \ref{binomial hermite 1} from Lemma~\ref{lem:first lemma}.
The right hand side equals $\binom{k}{j}(-1)^j(4r)^{k-j}$, which is $F_{j,k}$ as desired.
The case where $k$ is odd is analogous.
This finishes proving (1).

Next we consider (2). We have
\begin{align*}
(ST)_{j,k}&= (-1)^{\frac{N-b}{2}+k}2^{2(k-j)}\sum_{t=0}^{N-1}(-\I)^t \dfrac{H_{t-N+2j+b}(r)}{(t-N+2j+b)!}\dfrac{H_{N-2k-b-t}(\I r)}{(N-2k-b-t)!}\\
&=(-1)^{\frac{N-b}{2}+k}2^{2(k-j)} \sum_{t=N-b-2j}^{N-a-2k} (-\I)^t \dfrac{H_{t-N+2j+b}(r)}{(t-N+2j+b)!}\dfrac{H_{N-2k-b-t}(\I r)}{(N-2k-a-t)!}\\
&=(-1)^{\frac{N-b}{2}+k}2^{2(j-k)} (-\I)^{N-b-2j}\frac{1}{(2(j-k))!}\sum_{t=0}^{2(k-j)}(-\I)^t \sbinom{2(j-k)}{t} H_t(r)H_{2(j-k)-t}(\I r),
\end{align*}
where in the second equality we used again that if either $t-N+2j+b<0$ or $N-2k-b-t<0$ then the sum vanishes.
In particular, if $j<k$ then $(ST)_{j,k}=0$.
For $j\geq k$, we apply \ref{binomial hermite 2} from Lemma~\ref{lem:first lemma} to obtain
\[(ST)_{j,k} = (-1)^{\frac{N-b}{2}+k}2^{2(k-j)}(-\I)^{N-b-2j}\frac{1}{(2(j-k))!}\delta_{k-j,0}=\delta_{k-j,0},\]
and the proof is complete.
\end{proof}

\begin{proof}[Proof of Proposition~\ref{propiedad de 2SC^2QT=A}]
Using the standard Hermite polynomial identities $H_l(z)=2zH_{l-1}(z)-2(l-1)H_{l-2}(z)$ and $H'_l(z)=2lH_{l-1}(z)$ we have 
\[\int_{r}^{\infty}dz\,H_{l}(z)e^{-z^2}=H_{l-1}(r)e^{-r^2}\]
for all $l\geq1$.
On the other hand, the product of two Hermite polynomials can be expressed \cite{andrews1999special} as
\begin{equation}\label{producto dos polinomios de hermite}
H_j(z) H_{k}(z)
= 
\sum_{l=0}^{\min\{j,k\}}2^ll!
\binom{j}{l}
\binom{k}{l}
H_{j+k-2l}(z).	
\end{equation}
Using these facts in \eqref{eq:matrix_F} we get
\begin{gather}
Q_{j,k}
=
C_j^2
\sum_{l=0}^{\min\{j,k\}}2^ll!
\sbinom{j}{l}	
\sbinom{k}{l}
H_{j+k-2l-1}(r)e^{-r^2}\quad\text{for $j\neq k$},\\
Q_{0,0}=\tfrac{\sqrt{\pi}}{2}\mbox{erfc}(r),\qquad
Q_{jj}=\tfrac{1}{2}\mbox{erfc}(r) +
C_j^2
\sum_{l=0}^{j-1}2^ll!
\sbinom{j}{l}
\sbinom{j}{l}
H_{2j-2l-1}(r)e^{-r^2}\quad\text{ for }j>0,
\end{gather}
where $\mbox{erfc}(x)=\frac{2}{\sqrt{\pi}}\int_{x}^\infty dz\ts e^{-z^2}$ is the \emph{complementary error function}.
Now define matrices
\[\widetilde{Q}=Q-\tfrac{1}{2}\mbox{erfc}(r)I\qquad\text{and}\qquad\tilde{A}=A-\tfrac{1}{2}\mbox{erfc}(r)\tilde{I}\]
where $I$ and $\tilde{I}$ are the identity matrices of sizes $N$ and $(N-a+1)/2$.
Note that
$\displaystyle
\tilde{A}_{j,k}=2\tilde{Q}_{2j+b-1,2k+b-1}
$. 
Thanks to (2) of Proposition \ref{propiedades_de_T_y_S}, in order to prove the result it is enough to show that
\[2S\tilde{Q}T=\tilde{A}.\]

We have
\begin{multline}\label{ecuacion}
\sqrt{\pi}e^{r^2}(S\tilde{Q}T)_{j,k}=(-1)^{\frac{N-b}{2}+k}2^{2(k-j)}
\!\!\!\!\sum_{t=N-2j-b}^{N-1}\sum_{l=0}^{t}  
\frac{1}{(t-N+2j+b)!}
2^{l-t}
\sbinom{t}{l}\frac{1}{(N-2k-b-l)!}\\
\times H_{t+2j+b-N}(r) (-\I)^l
\sum_{u=0}^{N-2k-b-l}
\sbinom{N-2k-b-l}{u}(-\I)^u
H_{N-2k-b-u-l}(\I r)
H_{t+u-l-1}(r),
\end{multline}
where we have applied the change $u\longmapsto u-l$.
We will focus on the case $j\geq k$ and explain at the end how the same arguments work for the case $j<k$.
Consider first the sum in $t$ restricted to $N-2k-b<t<N$.
Using \ref{binomial hermite 3} of Lemma~\ref{lem:first lemma} with $m=N-2k-b-l\geq0$ and $d=t+2k+b-1-N\geq 0$, the $u$ sum equals 
$\frac{\I^{N-2k-b-l}2^{N-2k-b-l}(t-l-1)!}{(t+2k+b-1-N)!}H_{t+2k+b-1-N}(r)$ for $l<t$ and is zero for $l=t$,
and thus \eqref{ecuacion} with the $t$ sum restricted to $N-2k-b<t<N$ is equal to
\begin{multline}
2^{N-2j-b}
\sum_{t=N-2k-b+1}^{N-1}
\frac{H_{t+2j+b-N}(r)H_{t+2k+b-1-N}(r) }{2^t(t-N+2j+b)!}
\sum_{l=0}^{t-1}(-1)^l
\sbinom{t}{l}
\sbinom{t-l-1}{N-2k-b-l}\\
=
2^{N-2j-b}
\sum_{t=N-2k-b+1}^{N-1}
\frac{H_{t+2j+b-N}(r)H_{t+2k+b-1-N}(r) }{2^t(t-N+2j+b)!},
\end{multline}
where we have used  \ref{binomial 1} in Lemma ~\ref{lem:second lemma} with $m=N-2k-b$ (which is even). Applying the change of variable $t\mapsto t-(N-2k-b+1)$,
using \eqref{producto dos polinomios de hermite} and changing the order of the sum, the previous expression is equal to 
\begin{multline}
\frac{1}{2^{2j-2k+1}}
\sum_{h=0}^{2k+b-2}
\sum_{t=h}^{2k+b-2}
\frac{1}{2^{t-h}(t+2j-2k+1-h)!}
\sbinom{t}{h}
H_{2t+2j-2k-1-2h}(r)\\
=
\frac{1}{2^{2j-2k+1}}
\sum_{t=0}^{2k+b-2}
\frac{1}{2^{t}}
\frac{1}{(t+2j-2k+1)!}
H_{2t+2j-2k-1}(r)
\sum_{h=0}^{2k+b-2-t}
\sbinom{t+h}{t},
\end{multline}
where in the second line we have changed $t\longmapsto t+h$ and interchanged the order of summation.
Using the identity $\sum_{i=t}^{s}\binom{i}{t}=\binom{s+1}{t+1}$, the $h$ sum above equals $\binom{2k+b-1}{2k+b-2-t}$, and then changing $t\longmapsto 2k+b-2-t$ we get
\begin{equation}\label{eq:first term}
\frac{1}{2^{2j+b-1}(2j+b-1)!}
\sum_{t=0}^{2k+b-2}
2^{t}t!
\sbinom{2j+b-1}{t}
\sbinom{2k+b-1}{t}
H_{2j+2k+2b-2t-3}(r).
\end{equation}

Now we compute the sum \eqref{ecuacion} in the region $N-2j-b \leq t \leq  N-2k-b$.
Consider first the terms with $0\leq l<t-1$.
For the sum in $u$ we use  \ref{binomial hermite 3} of Lemma \ref{lem:first lemma} as above, with the same choices of $m$ and $d$.
Now $m\geq0$ and $m+d=t-l-1\geq0$ but $d<0$, so the whole sum vanishes.
This means that we are only left with the term $l=t$, namely
\begin{multline}
(-1)^{\frac{N-b}{2}+k}2^{2(k-j)}
\sum_{t=N-2j-b}^{N-2k-b} 
\frac{1}{(t-N+2j+b)!}
\frac{1}{(N-2k-b-t)!}
H_{t+2j+b-N}(r)\\
\times\sum_{u=0}^{N-2k-b-t}  (-\I)^{u+t}
\sbinom{N-2k-b-t}{u}
H_{N-2k-b-u-t}(\I r)
H_{u-1}(r).
\end{multline}
The $u=0$ term vanishes since $H_{-1}(r)=0$.
Then changing variables $t\longmapsto t+N-2j-b$, the above becomes 
\[ (-1)^{k-j}2^{2(k-j)}
 \sum_{u=1}^{2j-2k} \frac{(-\I)^{u}H_{u-1}(r)}{u!(2j-2k-u)!}\sum_{t=0}^{2j-2k-u}
\sbinom{2j-2k-u}{t}(-\I)^{t}
H_{t}(r)
H_{2j-2k-u-t}(\I r).\]
By \ref{binomial hermite 2} of Lemma \ref{lem:first lemma} the last sum is equal to
$
\frac{1}{2^{2j-2k}(2j-2k)!}H_{2j-2k-1}(r)
$,
which completes the term for $t=2k+b-1$ for the sum
in \eqref{eq:first term}, then taking the change of variables $t\to 2k+b-1-t$ we obtain $\frac{1}{2}\sqrt{\pi}e^{r^2}\tilde{A}_{j,k}$ and the result follows.

The case $j<k$ is simpler.
In computing \eqref{ecuacion} we now have $t\geq N-2j-b+1>N-2k-b+1$, so only the first $t$ region needs to be considered.
The same argument now leads to \eqref{ecuacion} being equal to
\[
\frac{1}{2^{2j+b-1}(2j+b-1)!}
\sum_{t=0}^{2j+b-1}
2^{t}t!
\sbinom{2j+b-1}{t}
\sbinom{2k+b-1}{t}
H_{2j+2k+2b-2t-3}(r).
\]

\end{proof}

It only remains to prove Lemmas \ref{lem:first lemma} and \ref{lem:second lemma}.

\begin{proof}[Proof of Lemma ~\ref{lem:first lemma}]
Let $H_n^{[\alpha]}(x)=(\frac{\alpha}{2})^{n/2}H_n(x/\sqrt{2\alpha})$.
From \cite[Eqn. (4.2.1)]{roman_1984} (note that the book uses a different Hermite function normalization) we have
$\sum_{s=0}^{n}
\binom{n}{s}
H_{s}^{[\alpha]}(x)H_{n-s}^{[-\alpha]}(y)=(x+y)^n
$,
and therefore 
\[
\sum_{s\geq0}
\sbinom{m}{s}
H^{[-1/2]}_{m-s}(-r)
H^{[1/2]}_{s}(r)=0,
\]
while, since the $H_k(x)$ is even for even $k$ and odd for odd $k$,
\[
\sum_{s\geq0}
(-1)^s
\sbinom{m}{s}
H^{[-1]}_{m-s}(-r)
H_{s}(r)
=
\sum_{s\geq0}
\sbinom{m}{s}
H^{[-1]}_{m-s}(-r)
H_{s}(-r)
=
(-2r)^{m}.
\]	
Adding and subtracting these two equations we obtain $\sum_{s\geq0}\binom{m}{2s}H^{[-1]}_{m-2s}(-r)H_{2s}(r)=\frac{1}{2}(-2r)^{m}$ and $\sum_{s\geq0}\binom{m}{2s+1}H^{[-1]}_{m-2s-1}(-r)H_{2s+1}(r)=-\frac{1}{2}(-2r)^{m}$.
Since $H_n(r)=2^{n}H^{[1/2]}_n(r)$ and $H_n(\I r)=(-\I)^n2^{n}H^{[-1/2]}_n(-r)$, the identities in \ref{binomial hermite 1} follow.

Next we consider (2).
The case $m=0$ is straightforward (since $H_0\equiv1$).
For $m>0$ we use the identity 
\begin{equation}\label{eq:gamma-identity}
H_{n}(\gamma r)=\sum_{l=0}^{\lfloor\frac{n}{2}\rfloor}
\gamma^{n-2l}(\gamma^2-1)^l
\sbinom{n}{2l}
\frac{(2l)!}{l!}H_{n-2l}(r),
\end{equation}
which is valid for all $\gamma\in\mathbb{C}$ \cite{huang2000introduction}.
Using it with $\gamma=\I$ we get
\begin{multline}
\sum_{t=0}^{m}
(-\I)^{t}
\sbinom{m}{t}
H_{t}(r)
H_{m-t}(\I r)
=
\I^{m}
\sum_{t\geq0}\sum_{l\geq0}
(-1)^t
\sbinom{m}{t}
2^l
\sbinom{m-t}{2l}
\frac{(2l)!}{l!}H_{t}(r)H_{m-t-2l}(r),\\
=\I^{m}m!
\sum_{l\geq 0}\sum_{h\geq 0}
(-1)^t2^{l+h}
\frac{1}{l!h!}H_{m-2l-h}(r)
\frac{1}{(m-2l-2h)!}
\sum_{t\geq 0}(-1)^t
\sbinom{m-2l-2h}{t-h}=0,
\end{multline}
where in the second equality we used \eqref{producto dos polinomios de hermite} and the third one follows because the last sum equals $(-1)^h\sum_{t=0}^{m-2l-2h}(-1)^t\binom{m-2l-2h}{t}$, which vanishes by the Binomial Theorem.

For (3) we write use \eqref{eq:gamma-identity} and then \eqref{producto dos polinomios de hermite} to write the left hand side as
\begin{align}
\I^{m}
\sum_{u=0}^{m}&\sum_{l=0}^{\lfloor\frac{m-u}{2}\rfloor}
(-1)^u\sbinom{m}{u}
\sbinom{m-u}{2l}
\frac{2^l(2l)!}{l!}H_{m-u-2l}(r)H_{m+u+d}(r)\\
&=m!\tts\I^{m}
\sum_{u=0}^{m}\sum_{l=0}^{\lfloor\frac{m-u}{2}\rfloor}\sum_{h=0}^{m-u-2l}
\frac{(-1)^u(m+u+d)!2^{l+h}}{l!h!u!(m-2l-h-u)!(m+u+d-h)!}
 H_{2m+d-2l-2h}(r)\\
&=m!\tts \I^{m} \sum_{l=0}^{\lfloor\frac{m}{2}\rfloor}\sum_{h=0}^{m-2l} \frac{2^{l+h}H_{2m+d-2l-2h}(r)}{l!h!} \sum_{u=0}^{m-2l-h}
\frac{(-1)^u(m+u+d)! }{u!(m-2l-h-u)!(m+u+d-h)!}\\
&=m!(m+d)!\tts \I^{m} \sum_{l=0}^{\lfloor\frac{m}{2}\rfloor}
\sum_{h=0}^{m-2l} 
\frac{(-1)^{m-h}2^{l+h} H_{2m+d-2l-2h}(r)}{l!(m-2l-h)!(2h+2l-m)!(2m+d-2l-2h)!},
\end{align}

\noindent where we have used Lemma \ref{lem:second lemma} part \ref{binomial 2} with $k=m-2l-h$ and $t=m+d$.
Let $g(r)$ denote the last expression.
We want to prove that $g(r)=\frac{\I^m2^m(m+d)!}{d!}H_d(r)$ which, by the orthogonality of the Hermite polynomials 
($
\int_{\mathbb{R}}dr\,H_n(r)H_k(r)e^{-r^2}
=
C_n^{-2}\delta_{n,k}
$)
is equivalent to proving that, for each $n\geq0$,
\[\int_{\RR}dr\,g(r)H_n(r)e^{-r^2}=\ts\frac{\I^m2^m(m+d)!}{C_d^2\tts d!}\delta_{n,d}.\]
Passing the integration inside the sums defining $g(r)$ yields a factor $C_n^{-2}\delta_{n,2m+d-2l-2h}$ inside the sum, and thus the integral can only be non-zero if $d$ and $n$ have the same parity, in which case the only non-zero term comes from $h=m-l+\frac{d-n}{2}$.
In other words, the left hand side above equals
\[
\frac{m!(m+d)!\tts \I^{m}(-1)^{(n-d)/2}2^{m+(d-n)/2}}{C_n^2(m-(n-d))!(\frac12(n-d))!n!} \sum_{l=0}^{\lfloor\frac{m}{2}\rfloor}
(-1)^{l}\binom{\frac12(n-d)}{l}.
\]
The prefactor vanishes unless $0\leq(n-d)\leq m$, and in this case we have $\lfloor m/2\rfloor\geq(n-d)/2$, so the last sum only ranges over $0\leq l \leq (n-d)/2$ and then if $n>d$ the sum vanishes by the Binomial Theorem. 
The only possibility is then $n=d$, which forces $l=0$ and $h=m$, turning the above expression into $\frac{\I^m2^m(m+d)!}{C_d^2\tts d!}$, as desired.
\end{proof}

\begin{proof}[Proof of Lemma \ref{lem:second lemma}]
Using the Binomial Theorem we have, for $x\in\RR$ with $|x|<1$,
\[\textstyle\sum_{m\geq0}\sum_{l\geq 0}(-1)^l 
\binom{t}{l}
\binom{t-l-1}{m-l}
x^m
=\sum_{l\geq0}(-x)^l\binom{t}{l}\sum_{m\geq 0}
\binom{t-l-1}{m}x^m
=(1+x)^{t-1}\sum_{l\geq 0}
\binom{t}{l}
\left(\tfrac{-x}{1+x}\right)^l,
\]
which equals $\frac{1}{1+x}=\sum_{m\geq0}(-1)^mx^m$.
(i) now follows by equating coefficients.

For (ii) it is enough to rearrange the identity as 
$
\sum_{u=0}^{k}(-1)^u
\binom{k}{u}\binom{t+u}{h}
=
(-1)^k
\binom{t}{h-k}
$,
which can be proved by induction in $k$ using Pascal's rule
$
\binom{k+1}{u}=\binom{k}{u}+\binom{k}{u-1}.
$
\end{proof}

\vskip6pt

\noindent{\bf Acknowledgements.}
Both authors received support from Centro de Modelamiento Matemático Basal Funds FB210005 from ANID-Chile.
YY’s work was funded in part by the National Agency for Research and Development (ANID)/ DOCTORADO NACIONAL/2019-21191420. Both authors thank Professor Daniel Remenik for many helpful discussions and comments and for his exceptional support and guidance through the development of this article. 


\bibliographystyle{alpha}
\bibliography{refs}

\end{document}